\documentclass[a4paper]{article}
\usepackage{enumerate}
\usepackage{enumitem}
\usepackage{epsfig,amssymb,amsmath,amsthm,pstricks}

\newcommand{\erdos}{Erd\H{o}s}
\newcommand{\real}{\mathbb{R}}

\newtheorem{defn_monotonic}{Definition}
\newtheorem{defn_mono_in_set}[defn_monotonic]{Definition}
\newtheorem{propn_nonstrict_set_bound}[defn_monotonic]{Proposition}
\newtheorem{defn_strict_monotonic}[defn_monotonic]{Definition}
\newtheorem{defn_strict_mono_in_set}[defn_monotonic]{Definition}
\newtheorem{thm_strict_set_bound}[defn_monotonic]{Theorem}
\newtheorem{thm_strict_seq_bound}[defn_monotonic]{Theorem}

\newtheorem{esconstruction}[defn_monotonic]{Lemma}
\newtheorem{sesconstruction}[defn_monotonic]{Lemma}

\newtheorem{defnintersectingpoints}[defn_monotonic]{Definition}
\newtheorem{thm_crossintersecting}[defn_monotonic]{Theorem}
\newtheorem{defnintersectingflats}[defn_monotonic]{Definition}
\newtheorem{defnminimalflats}[defn_monotonic]{Definition}
\newtheorem{intersectingflats}[defn_monotonic]{Lemma}
\newtheorem{intersectinggridsubset}[defn_monotonic]{Lemma}
\newtheorem{minantichainsize}[defn_monotonic]{Lemma}

\newtheorem{defnpowerfunction}[defn_monotonic]{Definition}
\newtheorem{defndomination}[defn_monotonic]{Definition}
\newtheorem{defnstable}[defn_monotonic]{Definition}

\newtheorem{stableupperbound}[defn_monotonic]{Theorem}
\newtheorem{stableupperbound3}[defn_monotonic]{Theorem}
\newtheorem{internallystableconstr}[defn_monotonic]{Theorem}

\begin{document}
\title{Strictly monotonic multidimensional sequences and stable sets in pillage games}

\author{David Saxton}

\maketitle

\begin{abstract}
Let $S \subset \real^n$ have size $|S| > \ell^{2^n-1}$. We show that there are distinct points $\{x^1, \ldots, x^{\ell+1}\} \subset S$ such that for each $i \in [n]$, the coordinate sequence $(x^j_i)_{j=1}^{\ell+1}$ is strictly increasing, strictly decreasing, or constant, and that this bound on $|S|$ is best possible.
This is analogous to the \erdos-Szekeres theorem on monotonic sequences in $\real$.

We apply these results to bound the size of a stable set in a pillage game.

We also prove a theorem of independent combinatorial interest. Suppose $\{a^1,b^1,\ldots,a^t,b^t\}$ is a set of $2t$ points in $\real^n$ such that the set of pairs of points not sharing a coordinate is precisely $\{\{a^1,b^1\},\ldots,\{a^t,b^t\}\}$. We show that $t \leq 2^{n-1}$, and that this bound is best possible.
\end{abstract}


\section{Introduction}

The main theorem of this paper is Theorem \ref{thm_strict_set_bound}, which concerns the existence of strictly monotonic sequences in $\real^n$ (for some definition of strictly monotonic). The proof of Theorem \ref{thm_strict_set_bound} also requires Theorem~\ref{thm_crossintersecting}, a theorem of independent interest. Section \ref{secn_pillage_games} describes an application of our results to stable sets in pillage games (this was the original motivation for Theorem \ref{thm_strict_set_bound}). We begin by giving some background.

\subsection{Non-strict monotonicity}
\label{secn_nonstrict}

A theorem of \erdos\ and Szekeres \cite{szekeres} tells us that within a sequence of $ab + 1$ real numbers, we can always find a monotonically increasing subsequence of length $a+1$ or a monotonically decreasing subsequence of length $b+1$. The bound $ab + 1$ is best possible, as can be seen by considering the sequence
\begin{equation}
 \label{real_sequence_es}
 (b,b-1,\ldots,1,2b,2b-1,\ldots,b+1,\ldots,ab,ab-1,\ldots,(a-1)b+1).
\end{equation}
The original proof of \erdos\ and Szekeres used geometrical reasoning. One can also deduce it from Dilworth's theorem (or an immediate corollary of it; see Lemma \ref{minantichainsize}) by considering a partial order where $x \leq y$ in the partial order if $x \leq y$ and $x$ occurs before $y$ in the sequence. Then a chain in this partial order corresponds to an increasing subsequence and an antichain corresponds to a decreasing subsequence. We also give a distinct proof due to Seidenberg \cite{seidenberg} below.

\begin{defn_monotonic}
 A sequence of points $(x^j)_{j=1}^{\ell}$ with $x^j \in \real^n$ is \emph{monotonic in direction} $c \in \{-1,1\}^n$ if for each $i \in [n]$, the $i$th coordinate sequence $(x_i^j)_{j=1}^\ell$ is (not necessarily strictly) decreasing or increasing according to whether $c_i = -1$ or $c_i = 1$ respectively.
\end{defn_monotonic}

We will sometimes omit the direction, so a monotonic sequence in $\real^n$ is one that is monotonic in some direction.

\begin{defn_mono_in_set}
 A set $S \subset \real^n$ contains a monotonic sequence of length $\ell$ (in direction $c$) if there are distinct points $\{x^1,\ldots,x^\ell\} \subset S$ such that the sequence $(x^j)_{j=1}^\ell$ is monotonic (in direction $c$).
\end{defn_mono_in_set}

There is a rough equivalence between sequences in $\real$ and sets in $\real^2$. A set in $\real^2$ can be ordered by the first coordinate (making an arbitrary choice of ordering when two points share a first coordinate) and projected in the second coordinate to get a sequence in $\real$. Conversely, a sequence of real numbers $(x^j)_{j=1}^\ell$ can be mapped to a set in $\real^2$ via $x^j \mapsto (j,x^j)$. These generalize to a rough equivalence between sequences in $\real^{n-1}$ and sets in $\real^n$. The \erdos-Szekeres theorem thus gives conditions guaranteeing a monotonic sequence in a set in $\real^2$.

If $(x^1,\ldots,x^\ell)$ is a monotonic sequence in direction $c$, then $(x^\ell,\ldots,x^1)$ is a monotonic sequence in direction $-c$, so for sequences in sets, we only need to consider one of $c$ or $-c$. With this in mind, define
$$ C_n = \{ c \in \{-1,1\}^n : c_1 = 1 \}, $$
and enumerate this set as $C_n = \{c^1,\ldots,c^{2^{n-1}}\}$.

The generalization of the \erdos-Szekeres theorem to $\real^n$ is as follows.

\begin{propn_nonstrict_set_bound}
 \label{propn_nonstrict_set_bound} (Non-strict monotonicity.)
Let $\ell_i \geq 1$ for $1 \leq i \leq 2^{n-1}$, and let $S \subset \real^n$ have size
\[ |S| > \prod_{i = 1}^{2^{n-1}} \ell_i. \]
Then $S$ contains a monotonic sequence of length $\ell_i+1$ in some direction $c^i$.
\end{propn_nonstrict_set_bound}

The bound in Proposition \ref{propn_nonstrict_set_bound} is best possible; we give a construction of size $\prod_i \ell_i$ containing no such sequence in Section \ref{secn_nonstrict_constr}. The proposition was proved by De Bruijin \cite{multidim} in the case $\ell_i = \ell$ for all $i$ by using $n-1$ applications of the \erdos-Szekeres theorem. For non-constant $\ell_i$, Proposition \ref{propn_nonstrict_set_bound} can be proved by a counting argument of Seidenberg \cite{seidenberg}, which we give below.

\begin{proof}
Let $S \subset \real^n$, $|S| = t$. Order the points by the first coordinate (so we consider $S$ as a sequence of points $(x^j)_{j=1}^t$ in $\real^n$). Assign each position $j, 1 \leq j \leq t$, a $2^{n-1}$-tuple of numbers $(r_1^j,\ldots,r_{2^{n-1}}^j)$, where $r_i^j$ is the maximum length of a subsequence in $S$ in direction $c^i$ ending at $x^j$. Then no tuple of numbers is repeated: given sequence positions $1 \leq j < k \leq t$, the point $x^k$ must lie in some direction $c^i$ from the point $x^j$. Then the sequence in direction $c^i$ of length $r_i^j$ ending at $x^j$ can be extended to a sequence containing $x^k$, and thus $r_i^j < r_i^k$. If no sequence has length $\ell_i + 1$ then $r_i^j \leq \ell_i$ for all $j$, and so by distinctness of the tuples, we have $|S| = t \leq \prod_i \ell_i$.
\end{proof}

\subsection{Strict monotonicity}
\label{secn_strict_sets}

Suppose we wish to find a strictly increasing, strictly decreasing or constant subsequence in a sequence in $\real$ (we must allow constant subsequences). The Seidenberg counting argument shows that a sequence in $\real$ with no such subsequence of length $\ell+1$ has maximum length $\ell^3$. This is best possible; consider the example (\ref{real_sequence_es}) with each $x$ replaced by $\ell$ consecutive copies of $x$.

\begin{defn_strict_monotonic}
 A sequence of points $(x^j)_{j=1}^{\ell}$ with $x^j \in \real^n$ is \emph{strictly monotonic in direction} $d \in \{-1,0,1\}^n$ if for each $i \in [n]$, the $i$th coordinate sequence $(x_i^j)_{j=1}^\ell$ is strictly decreasing, constant, or strictly increasing according to whether $d_i = -1$, $0$ or $1$ respectively.
\end{defn_strict_monotonic}

As before, we will sometimes omit the direction when talking about strictly monotonic sequences.

\begin{defn_strict_mono_in_set}
 A set $S \subset \real^n$ contains a strictly monotonic sequence of length $\ell$ (in direction $d$) if there are distinct points $\{x^1,\ldots,x^\ell\} \subset S$ such that the sequence $(x^j)_{j=1}^\ell$ is strictly monotonic (in direction $d$).
\end{defn_strict_mono_in_set}

We need consider only one of each $d$ or $-d$ for each $d \in \{-1,0,1\}^n$, so define
$$ D_n = \{ d \in \{-1,0,1\}^n : d \neq (0,\ldots,0), d_{i_0} = 1 \mbox{ where } i_0 = \min \{ i : d_i \neq 0 \} \}. $$

Consider a set $S \subset \real^n$. If we order the points by a coordinate to get a sequence $S'$ in $\real^{n-1}$, then a \emph{strictly} monotonic subsequence of $S'$ does not necessarily correspond to a strictly monotonic sequence in $S$ (it now matters what happens to points sharing the coordinate that we order by). Thus one cannot apply the counting argument of Seidenberg to bound the size of a set $S \subset \real^n$ with no strictly monotonic sequence of length $\ell+1$. Further, even if we start with a sequence in $\real^n$, the counting argument only gives a bound of $\ell^{3^n}$, which is far from best possible. Thus we need new techniques to work with strict monotonicity.

\subsection{Strict monotonicity in sets}

Suppose we wish to construct a large set in $\real^n$ with no strictly monotonic sequence of length $\ell+1$. Call such a set a \emph{good} set. Here we describe a natural construction which is in fact largest possible (see Section \ref{secn_strict_constr} for the exact construction).
As mentioned for Proposition \ref{propn_nonstrict_set_bound}, there is a set $E_n \subset \real^n$ of size $|E_n| = \ell^{2^{n-1}}$ with no monotonic sequence of length $\ell+1$, and so it is certainly good. In fact the construction in Section \ref{secn_nonstrict_constr} for $E_n$ contains no pair of points that share a coordinate for any coordinate position. Suppose $F_n$ is another good set such that every pair of points in $F_n$ share a coordinate in some coordinate position. If we replace each point in $E_n$ with a very small copy of $F_n$ to get a new set $G_n$ ($G_n$ is the ``product'' of $E_n$ and $F_n$; this is made more precise in Section \ref{secn_strict_constr}), then any strictly monotonic sequence in $G_n$ must either have all the coordinate sequences non-constant (thus taking at most one point from each copy of $F_n$), or it must lie strictly inside some fixed copy of $F_n$. In the first case, it corresponds to some monotonic sequence in $E_n$, and thus has length at most $\ell$. In the second case it has length at most $\ell$ since $F_n$ is good. Thus $G_n$ is also good. One candidate for $F_n$ is given by the recursive definition $F_1 = \{0\}$ and $F_n = f_n(G_{n-1})$, where $f_n : (x_1,\ldots,x_{n-1}) \mapsto (x_1,\ldots,x_{n-1},0)$. This recursive construction then gives $|G_n| = |E_n| |F_n| = \ell^{2^{n-1}} |G_{n-1}| = \ell^{2^n-1}$. Our main theorem shows that this is in fact best possible.

Now let $(\ell_d)_{d \in D_n}$ be a collection of maximal lengths with $\ell_d \geq 2$ for all $d$. We will show that the maximum size of a set with no strictly monotonic sequence in direction $d$ of length $\ell_d + 1$ for all $d$ is essentially the same recursive construction, with suitable choices at each stage to maximize the size of the set produced.

Let the function
$$ N : \{-1,0,1\}^n \rightarrow \mathcal{P}[n] $$
give the set of positions of the non-zero coordinates. Note that $|\{d \in D_n : N(d) = I\}| = 2^{|I|-1}$. Define $(m_I)_{I \subset [n]}$ and $(\lambda_I)_{I \subset [n]}$ via $\lambda_\emptyset = 1$ and
\begin{equation}
m_I = \prod_{d \in D_n : N(d) = I} \ell_d, \label{defn_mI}
\end{equation}
\begin{equation}
 \lambda_I = m_I \cdot \max_{i \in I} \lambda_{I \setminus \{i\}}. \label{defnlambda}
\end{equation}
For $I \subset [n]$, $\lambda_I$ should be thought of (this will be shown) as the maximum size of a good set $S$ when the coordinates of the points can only vary in $I$ (for all $x,y \in S$, $x_i = y_i$ for $i \notin I$). Similarly, $m_I$ should be thought of as the maximum size of a good set $S$ when the coordinates can only vary in $I$, and no two points share a common coordinate from $I$ (for all $x,y \in S$, $x_i = y_i$ if and only if $i \notin I$).

If $\ell_d = \ell$ for all $d \in D_n$, then $m_I = \ell^{2^{|I|-1}}$, $\lambda_{I} = \ell^{2^{|I|}-1}$ and $\lambda_{[n]} = \ell^{2^n-1}$.

\begin{thm_strict_set_bound}
 \label{thm_strict_set_bound} (Strict monotonicity in sets.)
Let $(\ell_d)_{d \in D_n}$ satisfy $\ell_d \geq 2$ for all $d$, and let $\lambda_{[n]}$ be as above. Let $S \subset \real^n$ have size
\[
|S| > \lambda_{[n]}.
\]
Then $S$ contains a strictly monotonic sequence of length $\ell_d+1$ in some direction $d \in D_n$.

In particular, if $|S| > \ell^{2^n-1}$ then $S$ contains a strictly monotonic sequence of length $\ell+1$.
\end{thm_strict_set_bound}

Theorem \ref{thm_strict_set_bound} is best possible; we give a construction of size $\lambda_{[n]}$ with no such sequence in Section \ref{secn_strict_constr}.
If we do not impose $\ell_d \geq 2$ for all $d$, then $\lambda_{[n]}$ may not be a correct bound. For example, with $n=3$, take the collection
$\ell_{(0,1,-1)} = \ell_{(1,0,-1)} = \ell_{(1,-1,0)} = 2$, and
$\ell_d = 1$ otherwise.
Then $\lambda_{[n]} = 2$, but
$ \{(1,0,0),(0,1,0),(0,0,1)\} $
is a good set of size 3.

\subsection{Strict monotonicity in sequences}
\label{secn_strict_monotonicity_in_sequences}

Theorem \ref{thm_strict_set_bound} bounds how large a \emph{set} can be without containing a strictly monotonic sequence. We would like an analogue of this theorem for sequences. Unlike the non-strict case, such an analogue is not a triviality.

Let $S = (x^j)_{j=1}^{|S|}$ be a sequence of points in $\real^n$. Each of the $3^n$ directions in $\{-1,0,1\}^n$ are now non-equivalent for the purposes of the existence of a subsequence in this direction. Suppose for each $d \in \{-1,0,1\}^n$, we forbid a subsequence of length $\ell_d + 1$ in direction $d$. Map the sequence $S$ to a set $T \subset \real^{n+1}$ as described in Section \ref{secn_nonstrict}, i.e., $x^j \mapsto (j,x^j)$. The set of maximum lengths for $T$ is now $(\ell_d^*)_{d \in D_{n+1}}$, where
\begin{equation*}
 \ell_{(d_0,d_1,\ldots,d_n)}^* =
\begin{cases}
 \ell_{(d_1,\ldots,d_n)} & \text{ if } d_0 = 1 \\
 1 & \text{otherwise (i.e., }d_0 = 0).
\end{cases}
\end{equation*}
Define $(\lambda_{I}^*)_{I \subset [n+1]}$ as in Section \ref{secn_strict_sets} for the lengths $(\ell_d^*)_{d \in D_{n+1}}$. We would like to apply Theorem \ref{thm_strict_set_bound}, but we cannot do this as stated, since we do not have $\ell_d^* \geq 2$ for all $d$. However, we will show that the proof of Theorem \ref{thm_strict_set_bound} still applies for this special case.

\begin{thm_strict_seq_bound}
 \label{thm_strict_seq_bound} (Strict monotonicity in sequences.)
Let $(\ell_d)_{d \in \{-1,0,1\}^n}$ satisfy $\ell_d \geq 2$ for all $d$, and let $\lambda_{[n+1]}^*$ be as above. Let $S$ be a sequence in $\real^n$ of length
\[ |S| > \lambda_{[n+1]}^*. \]
Then $S$ contains a strictly monotonic subsequence of length $\ell_d+1$ in some direction $d \in \{-1,0,1\}^n$.

In particular, if $|S| > \ell^{2^{n+1}-1}$ then $S$ contains a strictly monotonic subsequence of length $\ell+1$.
\end{thm_strict_seq_bound}

Theorem \ref{thm_strict_seq_bound} is best possible. The construction $G$ given in Section \ref{secn_strict_constr} for the set of lengths $(\ell_d^*)_{d \in D_{n+1}}$ has the property that no two points share the same first coordinate, so $G$ can be ordered by the first coordinate and projected in the remaining $n$ coordinates to get a sequence of points in $\real^n$.

\section{Lower bound constructions}
\label{secn_lower_bound_constrn}

\subsection{Construction for non-strict monotonicity}
\label{secn_nonstrict_constr}

Let $(\ell_i)_{i=1}^{2^{n-1}}$ be a collection of maximum lengths for the set of directions $C_n$. We will construct a set of size $\prod_i \ell_i$ with no sequence of length $\ell_i + 1$ in direction $c^i$, for all $1 \leq i \leq 2^{n-1}$. This shows the value appearing in Proposition \ref{propn_nonstrict_set_bound} is best possible.

For convenience, write
$$ L_k = \prod_{i=1}^k \ell_i $$
(and $L_0 = 1$). Define, for a set $A \subset \real^n$ and a vector $s \in \real^n$, the set translation
$$ A + s = \{ x + s : x \in A \}. $$
Define recursively, for $1 \leq k \leq 2^{n-1}$, the following collection of sets in $\real^n$.
\begin{align*}
 A_0 &= \{0\} \\
 A_{k,m} &= A_{k-1} + m L_{k-1} c^k\text{, \ for }0 \leq m \leq \ell_k - 1. \\
 A_k &= A_{k,0} \cup \cdots \cup A_{k,\ell_k-1}
\end{align*}
For $A \subset \real^n$, define
$$ w(A) = \max_{x,y \in A} ||x - y||_\infty, $$
where $||z||_\infty = \max_{1 \leq i \leq n} |z_i|$.

\begin{esconstruction}
 \label{esconstruction} (Properties of $A_k$)
\begin{enumerate}[label=(i)]
\item $A_k \subset \mathbb{Z}^n$,
\item $w(A_k) \leq L_k - 1$,
\item $|A_k| = L_k$,
\item for each $i$, $1 \leq i \leq k$, $A_k$ does not have a monotonic sequence of length $\ell_i+1$ in the direction $c^i$. Further $A_k$ does not have any non-trivial monotonic sequences in any of the directions $c^{k+1}, \ldots, c^{2^{n-1}}$.
\end{enumerate}
\end{esconstruction}

\begin{proof}
 (i) This is immediate from the construction.

 (ii) We proceed by induction on $k$. The statement is true for $A_0$. Let $x,y \in A_k$, say $x = x_0 + m_1 L_{k-1} c^k$, $y = y_0 + m_2 L_{k-1} c^k$, where $x_0, y_0 \in A_{k-1}$. Then
\begin{align*}
 ||x - y||_\infty & \leq ||x_0 - y_0|| + |m_1 L_{k-1} - m_2 L_{k-1}| \cdot || c^k ||_\infty \\
 & \leq L_{k-1} - 1 + (\ell_k-1) L_{k-1} \\
 & = L_k - 1.
\end{align*}

(iii) We proceed by induction on $k$. The statement is true for $A_0$. It is sufficient to show that $A_{k,m_1} \cap A_{k,m_2} = \emptyset$ for $m_1 \neq m_2$. Then, $|A_k| = \ell_k |A_{k-1}|$ and we are done. Indeed, suppose $x \in A_{k,m_1} \cap A_{k,m_2}$ for $m_1 \neq m_2$. Then $x - m_1 L_{k-1} c^k \in A_{k-1}$ and $x - m_2 L_{k-1} c^k \in A_{k-1}$. Hence
\begin{align*}
 w(A_{k-1}) & \geq || (x - m_1 L_{k-1} c^k) - (x - m_2 L_{k-1} c^k) ||_\infty \\
 & = |m_1-m_2| L_{k-1} \\
 & > L_{k-1} - 1,
\end{align*}
contradicting (ii).

(iv) We proceed by induction on $k$. The statement is true for $A_0$. Suppose $x \in A_{k,m_1}$, $y \in A_{k,m_2}$ with $m_1 < m_2$. Then again by (ii) we have that $(x,y)$ is a sequence in direction $c^k$. Therefore, if we have a sequence of points inside $A_k$, then either it must lie entirely inside an $A_{k,m_0}$ for some $m_0$ in direction $c^i$ for some $i < k$ (and thus have length at most $\ell_i$ by the inductive hypothesis); otherwise it lies in direction $c^k$ and can take at most one point from each $A_{k,m}, 0 \leq m \leq \ell_k-1$.
\end{proof}

Properties (iii) and (iv) of Lemma \ref{esconstruction} show that we can take the set $A_{2^{n-1}}$ as our construction.

\subsection{Construction for strict monotonicity}
\label{secn_strict_constr}

We construct sets showing that the bound is best possible for Theorem \ref{thm_strict_set_bound} (and Theorem \ref{thm_strict_seq_bound}) by induction on $n$. The case $n=1$ is simply $\ell_d$ distinct points in $\real$ for the unique $d \in D_1$.

Let $(\ell_d)_{d \in D_n}$, $(\lambda_I)_{I \subset [n]}$ be as in Theorem~\ref{thm_strict_set_bound}. In this section, a sequence will mean a strictly monotonic sequence. The set we construct will have size $\lambda_{[n]}$ with no sequence of length $\ell_d + 1$ in direction $d$ for all $d \in D_n$.

In equation (\ref{defn_mI}), $\{ d \in D_n : N(d) = [n] \} = C_n$, and so
$$ m_{[n]} = \prod_{c \in C_n} \ell_c. $$
Let $i_0 \in [n]$ be such that
$$ \lambda_{[n]} = m_{[n]} \lambda_{[n] \setminus \{i_0\} }. $$
Construct a new set of lengths $(\ell_d')_{d \in D_{n-1}}$ via
$$ \ell_d' = \ell_{(d_1, \ldots, d_{i_0-1}, 0, d_{i_0}, \ldots, d_{n-1})}. $$
By the inductive hypothesis with the collection $(\ell_d')_{d \in D_{n-1}}$ there is a set $F \subset \real^{n-1}$ of size $|F| = \lambda_{[n] \setminus \{i_0\}}$, which contains no sequence of length $\ell_d' + 1$ in direction $d$ for all $d \in D_{n-1}$. Let $F'$ be the embedding and scaling of $F$ into $\real^n$,
\begin{align*}
F' &= \{ (1/(w(F)+1))(x_1, \ldots, x_{i_0-1}, 0, x_{i_0}, \ldots, x_{n-1} ) : (x_1,\ldots,x_{n-1}) \in F \}.
\end{align*}
The set $F'$ has no sequences of length $\ell_d + 1$ in direction $d$ for all $d \in D_n \setminus C_n$, and no non-trivial sequences in any direction $d \in C_n$. Further $w(F') < 1$.

Construct another set of lengths $(\ell_i'')_{i=1}^{2^{n-1}}$ via
$$\ell_i'' = \ell_{c^i}.$$
Then as in Section~\ref{secn_nonstrict_constr}, there is a set $E \subset \real^n$ of size $|E| = \prod_i \ell_i'' = \prod_{c \in C_n} \ell_c = m_{[n]}$ with no sequence of length $\ell_c + 1$ in direction $c$ for all $c \in C_n$, and no non-trivial sequence in any direction $d \in D_n \setminus C_n$.

Lemma \ref{sesconstruction} shows that the following construction works.
\begin{equation}
 \label{defnG}
 G = \bigcup_{x \in E} (F' + x)
\end{equation}

\begin{sesconstruction}
 \label{sesconstruction} (Properties of $G$)
\begin{enumerate}[label=(i)]
 \item $|G| = \lambda_{[n]}$.
\item $G$ has no strictly monotonic sequence of length $\ell_d+1$ in direction $d$ for all $d \in D_n$.
\end{enumerate}
\end{sesconstruction}

\begin{proof}
(i) This argument is similar to that for property (iii) in Lemma \ref{esconstruction}. We have that $w(F') < 1$ and $||x - y||_\infty \geq 1$ for all $x,y \in E$, $x \neq y$, since $E \subset \mathbb{Z}^n$. Thus in (\ref{defnG}), $(F' + x) \cap (F' + y) = \emptyset$ for $x \neq y$, and so
$$ |G| = |F| |E| = \lambda_{[n] \setminus \{i_0\} } \cdot m_{[n]} = \lambda_{[n]}. $$

(ii) This argument is similar to that for property (iv) in Lemma \ref{esconstruction}. Let $(y^1, \ldots, y^t)$ be a strictly monotonic sequence in $G$ in direction $d \in D_n$, where $y^j \in F' + x^j$. There are two possibilities for this sequence. If it has a constant coordinate sequence, then the points must lie in some copy of $F'$, i.e., $x^1 = \cdots = x^t = x$ for some $x$, and $d_{i_0} = 0$. Then the construction of $F$ guarantees that $t \leq \ell_d$. Otherwise it has no constant coordinate, and so $(x^1,\ldots,x^t)$ is a sequence in $E$ in direction $d$, and so by construction of $E$, $t \leq \ell_d$.
\end{proof}

\section{Proof of the upper bound}
\label{secn_proof_of_the_upper_bound}

We begin by giving a theorem of independent interest, required for the proof of the main theorem.

\begin{defnintersectingpoints}
 Two points $x,y \in \real^n$ are \emph{intersecting} if they agree in some coordinate, i.e., $x_i = y_i$ for some $i \in [n]$.
\end{defnintersectingpoints}

\begin{thm_crossintersecting}
 \label{thm_crossintersecting}
 Let $\{a^1,b^1\},\ldots,\{a^t,b^t\}$ be a collection of $t$ pairs of points in $\real^d$ such that each pair $a^j,b^j$ is non-intersecting, but all $2t$ points are otherwise pairwise intersecting. Then $t \leq 2^{d-1}$.
\end{thm_crossintersecting}

This bound can be achieved by taking as pairs $\{a^j,b^j\}$ the opposing corners in the d-dimensional cube $\{0,1\}^d$.

Our proof uses exterior algebras, and is reminiscent of a proof of a theorem on intersecting sets given by Alon \cite{alon}. We describe them here briefly; for a comprehensive introduction the reader can consult for example Marcus \cite{marcus}. Given a real $n$-dimensional vector space $V$ with basis $\{e_1,\ldots,e_n\}$, the exterior algebra $\Lambda V$ is a $2^n$ dimensional vector space with basis $\{ e_A : A \subset [n] \}$ and an associative bilinear operation $\wedge$. For $A = \{i_1,\ldots,i_r \}, i_1 < \cdots < i_r$, we identify
$$ e_A = e_{i_1} \wedge \cdots \wedge e_{i_r}. $$
The operation $\wedge$ is defined to satisfy $e_i \wedge e_j = -e_j \wedge e_i$, and we extend by linearity. In particular, for a set of vectors $U = \{u_1, \ldots, u_m\} \subset V$, the wedge product $u_1 \wedge \cdots \wedge u_m$ is non-zero if and only if $U$ is an independent set of vectors.

\begin{proof}[Proof of Theorem \ref{thm_crossintersecting}]
Without loss of generality, the coordinates of the points take values in $[m]$, i.e., $\{a^1, b^1, \ldots, a^t, b^t\} \subset [m]^d$.

We consider the exterior algebra over the real vector space $\real^{2d}$. Label a set of basis elements for $\real^{2d}$ as
$$ \{e_1,\ldots,e_d,f_1,\ldots,f_d \}. $$
For each $i$, let $\text{lin}(e_i,f_i)$ be the subspace of $\real^{2d}$ spanned by the vectors $\{e_i,f_i\}$, and let
$$ \{ v_i^j : 1 \leq j \leq m \} \subset \text{lin}(e_i,f_i) $$
be a set of $m$ vectors in general position, i.e., any 2 of them are linearly independent. For $x \in [m]^d$, let $v_x$ be the vector
$$ v_x = \bigwedge_{i=1}^d v_i^{x_i}. $$
Then for $x,y \in [m]^d$, $v_x \wedge v_y = 0$ if and only if $x$ and $y$ intersect. Hence,
\begin{align*}
v_{a_i} \wedge v_{a_j} &= 0 \text{ \ for all } i,j, \\
v_{b_i} \wedge v_{b_j} &= 0 \text{ \ for all } i,j, \\
v_{a_i} \wedge v_{b_j} &= 0 \text{ \ if and only if } i \neq j.
\end{align*}
We now show that the vectors $\{ v_{a_j}, v_{b_j} : 1 \leq j \leq t \}$ are linearly independent. Suppose for some constants $\alpha_j, \beta_j$ we have that
$$ \sum_{j=1}^t \alpha_j v_{a_j} + \sum_{j=1}^t \beta_j v_{b_j} = 0. $$
For given $k$, the wedge product of the left hand side of this expression with $v_{a_k}$ is $\beta_k v_{b_k} \wedge v_{a_k}$. Since $v_{b_k} \wedge v_{a_k} \neq 0$, we must have $\beta_k = 0$. This is true for all $\alpha_k$ and $\beta_k$. This shows linear independence.

The vectors $v_{a_j}, v_{b_j}$ lie in the vector space spanned by the $2^d$ vectors of the form $x_1 \wedge \cdots \wedge x_d$, where $x_i \in \{e_i,f_i\}$. Thus by linear independence, we have $2t \leq 2^d$, i.e., $t \leq 2^{d-1}$.
\end{proof}

Here are two technical lemmas that we will need in the proof of Theorem \ref{thm_strict_set_bound}. Lemma \ref{minantichainsize} is also an immediate corollary of Dilworth's theorem.

\begin{intersectinggridsubset}
 \label{intersectinggridsubset}
 Let $m,d$ be integers, and let $A \subset [m]^d$ be intersecting (i.e., for all $x,y \in A$ there exists $i \in [d]$ such that $x_i = y_i$). Then $|A| \leq m^{d-1}$.
\end{intersectinggridsubset}

\begin{proof}
 Consider $[m]^d$ as the finite vector space $\mathbb{F}_m^d$. Partition $\mathbb{F}_m^d$ into the sets
$$\{(0,\ldots,0),(1,\ldots,1),\ldots,(m-1,\ldots,m-1)\} \subset \mathbb{F}_m^d $$
and all translates. There are $m^{d-1}$ such sets, and any intersecting subset of $\mathbb{F}_m^d$ takes at most one point from each set.
\end{proof}

\begin{minantichainsize}
 \label{minantichainsize}
Let $S$ be a partially ordered finite set, with maximum chain length $\ell$. Then $S$ contains an antichain $A$ of size at least $|A| \geq |S| / \ell$.
\end{minantichainsize}

\begin{proof}
The set $T$ of maximal elements in $S$ is an antichain, and the set $S \setminus T$ has maximal chain length $\ell-1$. Proceeding by induction on $\ell$ with the set $S \setminus T$ gives a partition of $S$ into $\ell$ antichains, one of which has size at least $|S| / \ell$.
\end{proof}

\subsection{Intersecting flats}

We consider \emph{axis-aligned affine subspaces} in $\real^n$, which we will refer to as \emph{flats}. Label such a flat $u$ via a string of $n$ numbers and $\star$s, say, $u \in (\real \cup \{\star\})^n$, where $\star$ is a wildcard. Thus a point $x \in \real^n$ lies in $u$ if for all $i \in [n]$, either $x_i = u_i$ or $u_i = \star$. For example, $(2,1,\star)$ represents a line in $\real^3$ parallel to the $z$-axis through the point $(2,1,0)$. Observe that the dimension of the flat is the number of $\star$ coordinates.

\begin{defnintersectingflats}
 A set of flats $W \subset (\real \cup \{\star\})^n$ is \emph{intersecting} if for all $u,v \in W$ we have $u_i = v_i \neq \star$ for some $i \in [n]$.
\end{defnintersectingflats}

Equivalently, a set of flats $W$ is intersecting if any pair of points taken from a single flat or a pair of flats in $W$ intersect.

\begin{defnminimalflats}
 \label{defnminimalflats}
An intersecting set of flats $W \subset (\real \cup \{\star\})^n$ is \emph{minimal} if no flats in $W$ can be enlarged (by replacing a coordinate with a $\star$) while $W$ remains intersecting.
\end{defnminimalflats}

Equivalently, $W$ is intersecting and minimal if $W$ is intersecting and further for all $u \in W$ and all $i \in [n]$ with $u_i \ne \star$, there is some $v \in W$ such that $u_i = v_i$ and for $j \ne i$, either $u_j \ne v_j$ or $u_j = v_j = \star$.

Given a set of pairwise intersecting points $V \subset \real^n$, we can construct an intersecting and minimal set of flats containing all the points of $V$ as follows. Initially, let $W = V$, considering $W$ as an intersecting set of flats (each of dimension $0$). Either $W$ is minimal, or we can enlarge one of the flats by replacing a coordinate with a $\star$, taking only one copy if this produces a duplicate flat. Continue in this manner until no further enlargements can be made.

The set of 3 flats listed below is an example of such an intersecting and minimal system in $\real^4$.
\begin{center}
\begin{tabular}{llll}
1 & $\star$ & 1 & $\star$\\
$\star$ & 0 & 1 & $\star$\\
1 & 0 & $\star$ & $\star$
\end{tabular}
\end{center}

We can use Theorem \ref{thm_crossintersecting} to bound the number of non-$\star$ values appearing in each coordinate position for a minimal set of intersecting flats.

\begin{intersectingflats}
 \label{intersectingflats}
 Let $W$ be an intersecting and minimal set of flats in $\real^n$. Then for all $i \in [n]$,
$$ | \{ w_i : w_i \neq \star, w \in W \} | \leq 2^{n-2}. $$
\end{intersectingflats}

\begin{proof}
Without loss of generality, we will bound the number of values occuring in the first coordinate. We may assume that the values occuring here are $[t]$. For each $j \in [t]$, there are $u^j, v^j \in W$ such that $u^j_1 = v^j_1 = j$, but either $u^j_i \neq v^j_i$ or $u^j_i = v^j_i = \star$ for $2 \leq i \leq n$. Project the last $n-1$ coordinates of $u,v$ to form $a^j, b^j \in \real^{n-1}$ respectively, where we replace $\star$s in $u$ with 1, and $\star$s in $v$ with 2. (There is nothing special about these two values other than that they are distinct.) Then $a^j$ and $b^j$ do not intersect. However, for $j \neq k$, $a^j,b^j$ must intersect $a^k,b^k$ since $W$ is intersecting and $u^j_1, v^j_1 \ne u^k_1, v^k_1$. The collection of pairs $\{a^1,b^1\},\ldots,\{a^t,b^t\}$ satisfies the conditions in Theorem \ref{thm_crossintersecting} and the result follows.
\end{proof}

\subsection{Proofs of the main theorems}

\begin{proof}[Proof of Theorem \ref{thm_strict_set_bound}]
We proceed by induction on $n$. The result is clearly true for $n=1$, so we may assume $n \geq 2$. Let $(\ell_d)_{d \in D_n}$, $(\lambda_I)_{I \subset [n]}$ be as in Theorem~\ref{thm_strict_set_bound}. Let $U \subset \real^n$ be a set of points containing no strictly monotonic sequence of length $\ell_d + 1$ in direction $d$, for all $d \in D_n$. We aim to show that $|U| \leq \lambda_{[n]}$.

Each direction $c \in C_n$ corresponds to a partial order on $U$, where we say $x < y$ if $(x,y)$ is a strictly monotonic sequence in direction $c$. The maximum chain length in the order corresponding to $c$ is $\ell_c$. Recursively construct a sequence of sets $U_0, \ldots, U_{2^{n-1}}$ as follows. Let $U_0 = U$. For $1 \leq i \leq 2^{n-1}$, partially order $U_{i-1}$ with $c^i$. Lemma \ref{minantichainsize} guarantees an antichain $U_i \subset U_{i-1}$ of size $|U_i| \geq |U_{i-1}| / \ell_{c^i}$. Set $V = U_{2^{n-1}}$. No two points in $V$ form a sequence in any of the directions $c \in C_n$, and so every pair of points agree in some coordinate ($V$ is pairwise intersecting). We have
$$ |V| \geq |U| / \prod_{c \in C_n} \ell_c = |U| / m_{[n]}. $$
It remains to show that
$$ |V| \leq \max_{i \in [n]} \lambda_{[n] \setminus \{i\}} = \lambda_{I_0} $$
(for some $I_0 \subset [n]$).

Define a function
$$ F : ( \real \cup \{\star\} )^n \rightarrow \mathcal{P}[n] $$
to give the coordinate positions of the $\star$-coordinates of a flat, i.e., $i \in F(w)$ if and only if $w_i = \star$.

Our inductive hypothesis tells us the following. Let $w \in (\real \cup \{\star\})^n$ be a flat and let $J = F(w) \subset [n]$ be the set of $\star$-coordinates of $w$. Let $X \subset V$ be the points of $V$ that lie in the flat $w$. Then $|X| \leq \lambda_{J}$. Indeed, let $t = |J|$ and let $\pi_J : \real^n \rightarrow \real^t$ project in the coordinate positions $J$. For $d \in D_t$, let $g(d) \in D_n$ be the direction such that $\pi_J(g(d)) = d$ and $g(d)_i = 0$ for $i \not\in J$. Define the set of lengths $(\ell_d')_{d \in D_t}$ via
$$\ell_d' = \ell_{g(d)}.$$
This gives a collection $(\lambda_I')_{I \subset [t]}$ with $\lambda_{[t]}' = \lambda_J$. Then by the inductive hypothesis with this collection of lengths, $|X| = |\pi_J(X)| \leq \lambda_{[t]}' = \lambda_J$.

In particular, let $W$ be an intersecting and minimal set of flats that contain all the points of $V$. Then
\begin{equation}
 \label{eqn_sum_of_lambda}
 |V| \leq \sum_{w \in W} \lambda_{F(w)}.
\end{equation}
It is sufficient to show that the right hand side of (\ref{eqn_sum_of_lambda}) is at most $\lambda_{I_0}$.
If all the points of $V$ lie in a flat of dimension $n-1$, i.e., all points share a single fixed coordinate position $i_0$, then $|W| = 1$ and we are done by the maximality of $\lambda_{I_0}$. So we will assume from now on that there is no single fixed coordinate, and all the flats have dimension at most $n-2$.

Let $J \subset [n]$ be a set of $|J| = s$ coordinate positions. Take a chain of subsets $J = J_s \subset J_{s+1} \subset \cdots \subset J_{n-1} \subset [n]$ with $|J_t| = t$. Since $\ell_d \geq 2$ for all $d \in D_n$, we have that $m_{J_t} \geq 2^{2^{t-1}}$ and hence
\begin{equation}
\label{lambdaforellequals2}
 \lambda_{I_0} \geq
\lambda_{J_{n-1}} \geq
\lambda_{J_{n-2}} 2^{2^{n-2}} \geq
\cdots \geq
\lambda_{J} 2^{2^{n-2} + 2^{n-3} + \cdots + 2^s} =
\lambda_{J} 2^{2^{n-1}-2^s}.
\end{equation}
Thus $\lambda_{J} / \lambda_{I_0} \leq 2^{2^s-1} / 2^{2^{n-1}-1}$ ($2^{2^s-1}$ is the value of $\lambda_{J}$ when $\ell_d = 2$ for all $d$). Thus from (\ref{eqn_sum_of_lambda}), it is sufficient to prove the inequality
\begin{equation}
 \label{eqn_sum_of_lambda_2}
f(W) := \sum_{w \in W} 2^{2^{|F(w)|}-1} \leq 2^{2^{n-1}-1}.
\end{equation}

We consider the specific cases $2 \leq n \leq 4$, before proving (\ref{eqn_sum_of_lambda_2}) for general $n \geq 5$. The cases of small $n$ are necessary to consider as we are inducting on $n$.

\paragraph {$n=2$.} By assumption $W$ consists only of flats of dimension $n-2$, i.e., points. By Lemma \ref{intersectingflats}, at most one value appears in each coordinate. However $W$ is minimal (Definition \ref{defnminimalflats}), so there are at least two values appearing in each coordinate. This contradiction implies that $W$ cannot exist (the only valid $W$ is one containing a single 1-dimensional flat).

\paragraph {$n=3$.} We aim to show that $f(W) \leq 2^{2^2-1} = 8$. By Lemma \ref{intersectingflats}, there are at most 2 non-$\star$ values used in each coordinate in $W$, say, $\{1,2\}$. If all flats have dimension 0, then $f(W) = |W| \leq 2^3$ and we are done.

Otherwise, say, $W$ contains the flat $(1,1,\star)$. By minimality, $W$ must contain a flat of the form $(1,a,c)$ and a flat of the form $(b,1,d)$, with $a,b \ne 1$, and no other forms are possible. Since $W$ is intersecting, $c=d \in \{1,2\}$. Further by minimality, $a = b = \star$. Thus $f(W) = 2|W| = 6$.

\paragraph {$n=4$.} We aim to show that $f(W) \leq 2^{2^3-1} = 128$. By Lemma \ref{intersectingflats}, there are at most 4 non-$\star$ values used in each coordinate in $W$, say, $\{1,2,3,4\}$.

Suppose first that there are only flats of dimension 0 or 1. Construct a bipartite graph $G$ with vertex classes $(X,Y)$ as follows. Let $X \subset \{1,2,3,4,\star\}^4$ be the set of all flats with exactly one $\star$ coordinate. Let $Y = \{1,2,3,4\}^4$. Add an edge between $x \in X$ and $y \in Y$ if $x$ and $y$ differ only in the unique $\star$-coordinate of $x$. $G$ is thus a 4-regular bipartite graph. The set of flats $W$ corresponds to an independent set of vertices in $G$. By Hall's Theorem, $G$ has a matching of size $|W|$, whose set of end vertices in $Y$ is an intersecting subset of $\{1,2,3,4\}^4$. This has maximum size $4^3$ by Lemma \ref{intersectinggridsubset}, hence $|W| \leq 4^3$ and $f(W) \leq 2|W| \leq 128$ as required.

So suppose there is a flat of dimension 2, say, $(1,1,\star,\star)$.
Similarly to the case $n=3$, by minimality $W$ must contain flats of the form $(1,a,c,d)$ and $(b,1,e,f)$ with $a,b \ne 1$, and no other flats are possible. Hence further by minimality, $a=b=\star$.
Denote the set of flats of the first form by $W_1 \subset W$ and those of the second form by $W_2 \subset W$, i.e.,
\begin{align*}
 W_1 &= \{ (1,\star,a,b) : a,b \in \{1,2,3,4,\star\} \} \cap W, \\
 W_2 &= \{ (\star,1,a,b) : a,b \in \{1,2,3,4,\star\} \} \cap W.
\end{align*}
So $W = \{(1,1,\star,\star)\} \cup W_1 \cup W_2$.
The number of 1-dimensional flats in each of $W_1$ and $W_2$ is at most $4^2$.
Suppose $(1,\star,a,\star)$ is a 2-dimensional flat in $W_1$. Then $a$ can take at most 1 value, since otherwise the flats of this form cannot all intersect $W_2$ ($W_2$ is non-empty). The same is true for flats of the form $(1,\star,\star,a)$, so there are at most two 2-dimensional flats in $W_1$, and similarly at most two 2-dimensional flats in $W_2$.
Thus
$$ f(W) \leq 2^{2^2-1} + 2 \cdot 4^2 \cdot 2 + 4 \cdot 2^{2^2-1} = 104 \leq 128. $$

\paragraph {$n\geq5$.}
Let $I \subset [n]$ be a set of size $|I| = s$. By Lemmas \ref{intersectinggridsubset} and \ref{intersectingflats} the number of flats $w \in W$ with $F(w) = I$ is at most $2^{(n-2)(n-s-1)}$. Thus
$$ f(W) \leq \sum_{s=0}^{n-2} \binom{n}{s} \cdot 2^{(n-2)(n-s-1)} \cdot 2^{2^s-1}. $$
It is easy to check that this quantity is at most $2^{2^{n-1}-1}$ for $n \geq 5$. Indeed one can calculate explicitly the values for $n=5,6,7$. For $n \geq 8$, bound the expression by $2^n \cdot 2^{(n-2)(n-1)} \cdot 2^{2^{n-2}-1}$. Looking at exponents, it is sufficient to show that $n^2 - 2n + 2 \leq 2^{n-2}$. This is true for $n=8$, and the right hand side grows faster than the left for all $n \geq 8$.
\end{proof}

\begin{proof}[Proof of Theorem \ref{thm_strict_seq_bound}]
Let $S$ be a sequence with no strictly monotonic subsequence of length $\ell_d + 1$ ($\ell_d \geq 2$) in direction $d$, for all $d \in \{-1,0,1\}^n$. Map it to a set $T \subset \real^{n+1}$ as described in Section \ref{secn_strict_monotonicity_in_sequences}. The proof of Theorem \ref{thm_strict_set_bound} is still valid for the set $T$, with the possible exception of inequality (\ref{lambdaforellequals2}). However, the only $J$ that we need to check (\ref{lambdaforellequals2}) for are those that occur as $J = F(w)$ for some $w \in W$, where $W$ is the intersecting and minimal set of flats that cover $V$. Indeed, let $w \in W$ and $J = F(w)$. Then we must have $w_1 = \star$ since $T$ contains no two points sharing the same first coordinate (the first coordinate is the index of the point in $S$), and hence $1 \in J$. Let $J \subset J_{s+1} \subset \cdots \subset J_{n-1} \subset [n]$ be a chain with $|J_t| = t$, as before. Then since $1 \in J_t$ for all $t$, we again have $m_{J_t} \geq 2^{2^{t-1}}$, and the inequality holds.
\end{proof}

\section{Pillage games}
\label{secn_pillage_games}

Theorem \ref{thm_strict_set_bound} was first motivated by an application to economic theory. We describe it here because it has some combinatorial interest.
Jordan \cite{initialjordan} introduced the concept of a pillage game. The set of \emph{players} in the pillage game is the set $[n] := \{ 1, \ldots, n \}$. A \emph{coalition} is a subset $I \subset [n]$ of the players. The set of \emph{allocations} is an $(n-1)$-dimensional simplex
$$ A = \{x \in \real^n : x_i \geq 0 \mbox{ for all } i \in [n] \mbox{ and } \sum_{i=1}^n x_i = 1 \}. $$
A \emph{distribution of wealth} of the players is a point $x \in A$.
The \emph{power function}, defined below, gives the strength of a coalition of players who have a certain distribution of wealth.

\begin{defnpowerfunction}
 The \emph{power function} is a map $\pi : \mathcal{P}[n] \times A \rightarrow \real$ satisfying
\begin{enumerate}[label=(p.1)]
\item if $C \subset C'$ then $\pi(C',x) \geq \pi(C,x)$ for all $x \in A$;
\item if $x_i' \geq x_i$ for all $i \in C$ then $\pi(C,x') \geq \pi(C,x)$; and
\item if $C \ne \emptyset$ and $x_i' > x_i$ for all $i \in C$ then $\pi(C,x') > \pi(C,x)$.
\end{enumerate}
\end{defnpowerfunction}

The above axioms thus specify monotonicity conditions: (p.1) says that the power of a coalition does not decrease if new members are added; (p.2) says that the power of a coalition does not decrease if the wealth of some of the members is increased without decreasing the wealth of other members; (p.3) says that if the wealth of each member of a coalition is strictly increased then the power of the coalition must also strictly increase.

\subsection{Domination and stable sets}
\label{secn_domination_and_stable_sets}

\begin{defndomination}
 Let $x, x' \in A$. Define the sets $W = \{i : x_i' > x_i\}$ (winners) and $L = \{i : x_i' < x_i \}$ (losers). Then $x'$ \emph{dominates} $x$ if
$$ \pi(W,x) > \pi(L,x). $$
\end{defndomination}

The interpretation of this is that $x'$ dominates $x$ if the coalition of players whose wealth strictly increases in going from $x$ to $x'$ is more powerful than the coalition of players whose wealth strictly decreases in going from $x$ to $x'$, when the wealth distribution is $x$.

\begin{defnstable}
 A set $S \subset A$ is a \emph{stable set} if it satifies
\begin{itemize}
\item (internal stability) no element of $S$ is dominated by any other element of $S$; and
\item (external stability) each element of $A \setminus S$ is dominated by some element of $S$.
\end{itemize}
\end{defnstable}

Note that a stable set need not exist.

Jordan \cite{initialjordan} makes the following observation. Let $S \subset A$ be a stable set. Let $\{x^1, x^2, x^3, x^4\} \subset S$ be four points in the stable set, and for $k=1,2,3$ define
\begin{align*}
 W_k & = \{ i : x_i^{k+1} > x_i^k \}, \\
 L_k & = \{ i : x_i^{k+1} < x_i^k \}.
\end{align*}
We have by internal stability of $S$ that
\begin{itemize}
 \item $\pi(W_{k-1},x^k) \geq \pi(L_{k-1},x^k)$ for $k=2,3,4$ since $x^k$ does not dominate $x^{k-1}$; and
 \item $\pi(W_k,x^k) \leq \pi(L_k,x^k)$ for $k=1,2,3$ since $x^k$ does not dominate $x^{k+1}$.
\end{itemize}
In particular we have that
\begin{align*}
\pi(W_1,x^2) &\geq \pi(L_1,x^2), &
\pi(W_2,x^2) &\leq \pi(L_2,x^2), \\
\pi(W_2,x^3) &\geq \pi(L_2,x^3), &
\pi(W_3,x^3) &\leq \pi(L_3,x^3).
\end{align*}
If $W_1 = W_2 = W_3 = W$ and $L_1 = L_2 = L_3 = L$ then
\begin{align*}
\pi(W,x^2) &= \pi(L,x^2), &
\pi(W,x^3) &= \pi(L,x^3).
\end{align*}
However this violates axiom (p.3) of the power function, since we must have $\pi(W,x^2) < \pi(W,x^3)$ and $\pi(L,x^2) > \pi(L,x^3)$.
Therefore, given four points in a stable set, we cannot have $W_1 = W_2 = W_3$ and $L_1 = L_2 = L_3$.
More generally, it can be seen that we cannot have $W_1 \subset W_2 \subset W_3$ and $L_1 \supset L_2 \supset L_3$ either \cite{kerberrowatpersonal}.
In his original paper \cite{initialjordan} Jordan observed that this implies that any stable set must be finite.

In fact, the situation with $W_1 = W_2 = W_3 = W$ and $L_1 = L_2 = L_3 = L$ is precisely the case that $(x^1,x^2,x^3,x^4)$ is a strictly monotonic sequence in direction $d$, where
$$ d_i =
\begin{cases}
 1 & \text{if } i \in W, \\
 -1 & \text{if } i \in L, \\
 0 & \text{otherwise.}
\end{cases}
$$
(Pairs of disjoint sets $(W,L)$ and directions $d \in \{-1,0,1\}^n$ can both be used to specify an equivalent direction.)
So a stable set cannot contain a strictly monotonic sequence of length 4.

In $A$ any direction $(W,L)$ that appears between a pair of points must have both $W \ne \emptyset$ and $L \ne \emptyset$. There are $c = (3^n-2^{n+1}+1)/2$ pairs $\{W,L\}$ with $W,L \subset [n]$, $W \cap L = \emptyset$, and $W \neq \emptyset \neq L$. Kerber and Rowat \cite{ramseybound} used Ramsey theory to bound the maximum size of a stable set $S$, as follows. Consider the complete graph $K_{|S|}$ on $|S|$ vertices. Associate with each of the $c$ possible directions in $S$ a colour $i$, $1 \leq i \leq c$. Colour the edge $xy$, $x,y \in S$ with the colour associated with the direction of $(x,y)$ or $(y,x)$. Then $|S| < R_c(4)$, the maximum size of a complete graph coloured with $c$ colours with no monochromatic clique of size 4. Indeed, such a monochomatic clique would give a strictly monotonic sequence of length 4, which cannot exist by the observation above.

We have the bounds $3^c < R_c(4) \leq c^{2c+1}$ (the lower bound is a product construction that can be found in \cite{ramsey_lower_bound}. The upper bound follows from a pigeonhole argument; see for example \cite{ramsey_upper_bound}). The upper bound on $S$ can be improved if we make use of the fact that the colours in the graph are induced by the directions of a set of points in $\real^n$. Indeed, by Theorem \ref{thm_strict_set_bound} with $\ell=3$, we immediately have the following.

\begin{stableupperbound}
 \label{stableupperbound}
Let $S \subset A$ be a stable set for $n$ players. Then
$$ |S| \leq 3^{2^n-1}. $$
\end{stableupperbound}

In fact with a little more work one can tighten this bound to $3^{2^n-n-2} \cdot 2^n + n$. We demonstrate some of these ideas for $n=3$ in the next theorem.

\begin{stableupperbound3}
 \label{stableupperbound3}
Let $S \subset A$ be a stable set for $3$ players. Then
$$ |S| \leq 27. $$
\end{stableupperbound3}

Kerber and Rowat \cite{threeagentpillage} have a tighter (in fact, best possible) bound of 15 under certain niceness conditions on the power function. Here we do not make any assumptions on the power function beyond the three axioms given.

\subsection{Towards constructing large stable sets}

Theorem \ref{stableupperbound} gives a super-exponential upper bound for the size of a stable set. In fact, the proof only makes use of internal stability. The next theorem gives a subset $S \subset A$ of a similar super-exponential size together with a power function defined on $S$ which makes $S$ into an internally stable set. We have not been able to determine whether or not the power function could be extended to one defined on all of $A$, while making $S$ into an externally stable set as well. However, this theorem at least shows that we cannot hope to improve substantially on the bound in Theorem \ref{stableupperbound} by considering internal stability alone.

\begin{internallystableconstr}
 \label{internallystableconstr}
Let $n$ be given, and let $d = \lfloor (n-1) / 2 \rfloor$. There exists a set $S \subset A$ of size
$$ |S| = 3^{\frac{1}{2} \binom{2d}{d}} $$
together with a function $\pi : \mathcal{P}[n] \times S \rightarrow \real$ satisfying axioms (p.1), (p.2), (p.3) that makes $S$ an internally stable set.
\end{internallystableconstr}

\subsection{Proof of Theorem \ref{stableupperbound3}}

\begin{proof}
Let $S \subset A$ be a stable set. Let $S' \subset S$ be $S$ minus any corner points that it may contain, i.e.,
$$ S' = S \setminus \{ (1,0,0), (0,1,0), (0,0,1) \}. $$

Suppose, say, that $c = (0,1,0) \not \in S$. Then by external stability, $c$ is dominated by some point $x \in S$, where $x_1,x_3 \geq 0$ and $x_2 < 1$. Let $F \subset \{1,3\}$, $F \ne \emptyset$ be such that $x_i > 0$ for $i \in F$. Then by the definition of domination,
\begin{equation}
 \label{pi_for_corner}
 \pi(F, c) > \pi(\{2\}, c).
\end{equation}
This condition then forbids any pair of points forming a sequence in direction $(1,-1,1)$ in $S'$. Indeed, suppose $y,z \in S'$ with $(y,z)$ such a sequence.
By axioms (p.2) and (p.1) we have that
\begin{align*}
 \pi( \{1,3\}, y ) & \geq \pi( \{1,3\}, c ) \geq \pi( F, c ), \\
 \pi( \{2\}, y )   & \leq \pi( \{2\}, c ).
\end{align*}
But then by (\ref{pi_for_corner}) we must have that $\pi( \{1,3\}, y ) > \pi( \{2\}, y )$, i.e., $z$ dominates $y$, contradicting the internal stability of $S$.

If we do have the point $c = (0,1,0)$ in $S$ then there can be no 3 points in a sequence in $S'$ in the direction $(1,-1,1)$. Indeed, suppose $x,y,z$ formed a sequence in the direction $(1,-1,1)$, or equivalently, the direction $(\{1,3\},\{2\})$. Then $(c,x)$ is a sequence in the direction $(W,\{2\})$, where $W \subset \{1,3\}$. By the observation in Section \ref{secn_domination_and_stable_sets}, $S$ cannot then be a stable set.

So in each of the three directions $(1,1,-1),(1,-1,1),(-1,1,1),$ the maximum length of a sequence of points in $S'$ is either 1 or 2. As in the proof of Theorem \ref{thm_strict_set_bound} we can find a subset $|S''| \geq |S'| / 2^3$ with no pair of points lying in any of these directions. This means that for all $x,y \in S''$ we have $x_i = y_i$ for some $i \in [3]$.

We now show that $|S''| \leq 3$. If all the points in $S''$ share some coordinate, say $x_1 = a$ for all $x \in S''$, then the points in $S''$ form a sequence in direction $(0,1,-1)$, and so $S'' \leq 3$. Otherwise, for some $a,b,c \in [0,1]$, $S''$ contains the points
$$
 ( a     , b , 1-a-b ),
 ( a     , c , 1-a-c ),
 ( a+c-b , b , 1-a-c ),
$$
with $b \ne c$. No further distinct points can be added that share a coordinate with each of the three points above. Thus again $|S''| \leq 3$.

Putting these results together, we have that
$$ |S| \leq 3 + |S'| \leq 3 + 2^3 |S''| \leq 27. $$
\end{proof}

\subsection{Proof of Theorem \ref{internallystableconstr}}

\begin{proof}
 We construct such a subset of $A \subset \real^n$.
For $c \in C_{2d}$, let $p(c)$ count the number of $1$-coordinates of $c$. Define a collection of lengths $(\ell_i)_{i=1}^{2^{2d-1}}$ via
$$ \ell_i = 
\begin{cases}
 3 & \text{if } p(c^i) = d, \\
 1 & \text{otherwise.}
\end{cases} $$
By the construction in Section \ref{secn_nonstrict_constr} there is a set $T \subset \real^{2d}$ of size $\prod_i \ell_i = 3^{\frac{1}{2} \binom{2d}{d}}$ with no sequence of length $\ell_i+1$ in direction $c^i$. Moreover, for every two distinct points $x,y \in T$, $x_i \ne y_i$ for all $i \in [2d]$. In particular, every pair of points in $T$ must lie in some direction $c \in C_{2d}$ with $p(c)=d$.

Scale and translate the set so that it lies in $[0,1/(2d)]^{2d} \subset \real^{2d}$. Then map it to a set $S$, $S \subset A \subset \real^n$, by adding either one or two last coordinates (depending on whether $n = 2d+1$ or $n = 2d + 2$) such that the sum of the coordinates of any point is $1$.

We now define a power function on $S$. Let $B \subset [2d]$ with $|B| = d$ and let $x \in S$, with $x$ the image of the point $x' \in T$. Partially order the points of $T$ such that $u < v$ if $u_i < v_i$ for all $i \in B$, so by the properties of $T$, $u_i > v_i$ for all $i \in [2d] \setminus B$. Note that $T$ contains no chain of length 4 by construction. Let $q(B,x)$ denote the furthest distance up a chain that $x'$ lies, i.e.,
$$ q(B,x) =
\begin{cases}
 3 & \text{if there are $u,v \in T$ with $u < v < x'$,} \\
 2 & \text{if otherwise there is some $u \in T$ with $u < x'$,} \\
 1 & \text{otherwise.}
\end{cases} $$
Now define $\pi : \mathcal{P}[n] \times S \rightarrow \real$ via
$$ \pi(C, x) =
\begin{cases}
 \max_{i \in C} x_i     & \text{if $|C \cap [2d]| < d$,} \\
 q(C \cap [2d],x)       & \text{if $|C \cap [2d]| = d$,} \\
 \max_{i \in C} x_i + 3 & \text{if $|C \cap [2d]| > d$}
\end{cases} $$
where we take $\max_{i \in \emptyset} x_i = 0$.

This satisfies axiom (p.1), since $0 \leq \max_{i \in C} x_i \leq 1$ and $1 \leq q(B,x) \leq 3$.

Axioms (p.2) and (p.3) are satisfied when $|C \cap [2d]| \ne d$, as can easily be checked. So let $C$ be such that $|C \cap [2d]| = d$. Let $x,y \in S$, $x \ne y$, with $x_i \leq y_i$ for all $i \in C$. Then in fact $x_i < y_i$ for all $i \in C \cap [2d]$, and so $x < y$ in the partial order given by $C$. Thus $q(C \cap [2d], x) < q(C \cap [2d], y)$. Thus axioms (p.2) and (p.3) are satisfied in this case as well.

It remains to check internal stability, i.e., that no point in $S$ dominates any other point of $S$. Let $x, y \in S$ be two distinct points, with $W = \{ i : y_i > x_i \}$ and $L = \{ i : y_i < x_i \}$, and let $x', y'$ be their original points in $T$. By construction of $T$, $|W \cap [2d]| = |L \cap [2d]| = d$. In the partial order on $T$ induced by $W \cap [2d]$, as described above, we have that $x' < y'$, and so $q(W \cap [2d], x) \leq 2$. Similarly $q(L \cap [2d], x) \geq 2$. Thus $y$ does not dominate $x$, and we are done.
\end{proof}

\paragraph{Acknowledgements} The author would like to thank Manfred Kerber and Colin Rowat for drawing attention to this question and Andrew Thomason for his helpful suggestions and comments.

\bibliographystyle{elsarticle-num}
\bibliography{bib}
\end{document}